\documentclass[12pt, reqno]{amsart}
\usepackage{amsmath, amsthm, amscd, amsfonts, amssymb, graphicx, color}
\usepackage[bookmarksnumbered, colorlinks, plainpages]{hyperref}

\theoremstyle{plain}
\newtheorem{theorem}{Theorem}
\newtheorem*{main}{Main Theorem}
\newtheorem*{A}{Theorem A}
\newtheorem*{B}{Theorem B}
\newtheorem{lemma}{Lemma}[section]
\newtheorem{proposition}{Proposition}[section]
\newtheorem{corollary}{Corollary}

\theoremstyle{definition}

\newtheorem{remark}{Remark}

\numberwithin{equation}{section}

\begin{document}
\setcounter{page}{1}

\title[Essential spectrum of the sum of self-adjoint operators]
{On the essential spectrum of the sum of self-adjoint operators and the closedness of the sum of operator ranges}

\author[Ivan S. Feshchenko]{Ivan S. Feshchenko}

\address{Taras Shevchenko National University of Kyiv,
Faculty of Mechanics and Mathematics, Kyiv, Ukraine}

\email{\textcolor[rgb]{0.00,0.00,0.84}{ivanmath007@gmail.com}}



\date{Received: 26 December 2012; Accepted: 14 April 2013.}

\begin{abstract}
Let $\mathcal{H}$ be a complex Hilbert space, and $A_1,\ldots,A_N$ be bounded self-adjoint operators in $\mathcal{H}$
such that $A_i A_j$ is compact for any $i\neq j$.
It is well-known that $\sigma_e(\sum_{i=1}^N A_i)\setminus\{0\}=(\cup_{i=1}^N\sigma_e(A_i))\setminus\{0\}$, where
$\sigma_e(B)$ stands for the essential spectrum of a bounded self-adjoint operator $B$.

In this paper we get necessary and sufficient conditions for $0\in\sigma_e(\sum_{i=1}^N A_i)$.
This conditions are formulated in terms of the projection valued spectral measures of $A_i$, $i=1,\ldots,N$.
Using this result, we obtain necessary and sufficient conditions for the sum of ranges of $A_i$, $i=1,\ldots,N$
to be closed.

\end{abstract}

\maketitle

\section{Introduction}\label{S:introduction}

Let $\mathcal{H}$ be a complex Hilbert space, and
$A_1,\ldots,A_N$ be $N\geqslant 2$ bounded self-adjoint operators in $\mathcal{H}$
such that $A_i A_j$ is compact for $i\neq j$.
Define
\begin{equation*}
A=\sum_{i=1}^N A_i.
\end{equation*}

\subsection{On the essential spectrum of the sum of self-adjoint operators}

The following result on the essential spectrum of $A$, $\sigma_e(A)$, is well-known and has applications in scattering theory
and spectral analysis of Hankel operators (see, e.g.,~\cite[Proposition 3.1]{Pushnitski_Yafaev},\cite[Chapter 10, Lemma 1.5]{Peller}).

\begin{A}
We have
\begin{equation*}
\sigma_{e}(A)\setminus\{0\}=\left(\bigcup_{i=1}^N\sigma_{e}(A_i)\right)\setminus\{0\}.
\end{equation*}
\end{A}

For convenience of the reader, we will prove Theorem A (following~\cite{Pushnitski_Yafaev}) in Section~\ref{S:proof_theorem_A}.

What one can say about the point $0$?

If $\mathcal{H}$ is infinite dimensional, then $0\in\cup_{i=1}^N\sigma_e(A_i)$.
Indeed, it is easily checked, that if $B_1,B_2$ are bounded self-adjoint operators in $\mathcal{H}$ such that $B_1 B_2$ is compact,
then $0\notin\sigma_e(B_1)\Rightarrow B_2$ is compact.
Hence, if $0\notin\sigma_e(A_1)$, then $A_2$ is compact, and, consequently, $\sigma_e(A_2)=\{0\}$.

Thus, the following question arises naturally: when $0\in\sigma_e(A)$?
We will give an answer to this question in terms of the projection valued spectral measures of $A_i$, $E_{A_i}(\cdot)$, $i=1,\ldots,N$.
Define the subspace
\begin{equation}\label{E:H_varepsilon}
\mathcal{H}_\varepsilon=\mathcal{H}_\varepsilon(A_1,\ldots,A_N)=\bigcap_{i=1}^N E_{A_i}([-\varepsilon,\varepsilon])\mathcal{H},\quad
\varepsilon\geqslant 0.
\end{equation}

\begin{main}\label{T:main_theorem}
$0\in\sigma_{e}(A)$ if and only if
the subspace $\mathcal{H}_\varepsilon$ is infinite dimensional for any $\varepsilon>0$.
\end{main}

\subsection{On the closedness of the sum of operator ranges}

The problem on the closedness of the sum of operator ranges
(criteria for the sum of operator ranges to be closed,
properties of collections of operators $X_i$ such that $\sum_{i=1}^n Ran(X_i)$ is closed),
in particular,
the problem on the closedness of the sum of $n$ subspaces of a Hilbert space
(criteria for the sum of $n$ subspaces of a Hilbert space to be closed,
properties of collections of subspaces with closed sum),
is an important problem of functional analysis.
This problem was studied in numerous publications and has many applications in various branches of mathematics,
see, for example, \cite{Feshchenko} and the bibliography therein
(unfortunately, English translation of the Russian original of this paper in some places is bad.
For example, the word "closeness" must be replaced by "closedness"),~\cite{Deutsch},~\cite{Badea}.

We consider the following question: when the sum of ranges of $A_i$,
\begin{equation*}
\sum_{i=1}^N Ran(A_i)=\left\{\sum_{i=1}^N y_i\mid y_i\in Ran(A_i)\right\}=\left\{\sum_{i=1}^N A_i x_i\mid x_i\in\mathcal{H}\right\},
\end{equation*}
is closed?
This question arises naturally in connection with the following result having applications in theoretical tomography
(see \cite{Lang} and the bibliography therein).

\begin{B}
Let $P_1,\ldots,P_N$ be orthogonal projections in $\mathcal{H}$.
Suppose $P_i P_j$ is compact for $i\neq j$.
Then $\sum_{i=1}^N Ran(P_i)$ is closed.
\end{B}

\begin{remark}
Theorem B is stated only for Hilbert spaces, while
it was proved by Lars Svensson for reflexive Banach spaces~\cite{Svensson report},
by Harald Lang for Frechet spaces~\cite{Lang}, and
by Lars Svensson for Hausdorff locally compact topological vector spaces~\cite{Svensson article}.
\end{remark}

Using Main Theorem, we prove a criterion for $\sum_{i=1}^N Ran(A_i)$ to be closed.
Recall that the subspaces $\mathcal{H}_\varepsilon$, $\varepsilon\geqslant 0$, are defined by~\eqref{E:H_varepsilon}.
Clearly,
\begin{equation*}
\mathcal{H}_0=\bigcap_{i=1}^N Ker(A_i).
\end{equation*}
Note that $\mathcal{H}_0\subset\mathcal{H}_\varepsilon$ for any $\varepsilon>0$.

\begin{theorem}\label{T:criterion_sum_closed}
$\sum_{i=1}^N Ran(A_i)$ is closed if and only if $\mathcal{H}_\varepsilon=\mathcal{H}_0$ for some $\varepsilon>0$.
\end{theorem}

\begin{remark}
We will prove that if $\mathcal{H}_\varepsilon\ominus\mathcal{H}_0$ is finite dimensional for some $\varepsilon>0$, then
$\sum_{i=1}^N Ran(A_i)$ is closed.
\end{remark}

\begin{remark}
If $\sum_{i=1}^N Ran(A_i)$ is closed, then $\sum_{i=1}^N Ran(A_i)=\mathcal{H}\ominus\mathcal{H}_0$.
Indeed, suppose $\sum_{i=1}^N Ran(A_i)$ is closed.
We have
\begin{equation*}
\mathcal{H}\ominus\left(\sum_{i=1}^N Ran(A_i)\right)=\bigcap_{i=1}^N (\mathcal{H}\ominus Ran(A_i))=\bigcap_{i=1}^N Ker(A_i)=\mathcal{H}_0.
\end{equation*}
Hence, $\sum_{i=1}^N Ran(A_i)=\mathcal{H}\ominus\mathcal{H}_0$.
\end{remark}

Using this criterion for $\sum_{i=1}^N Ran(A_i)$ to be closed, we get the following generalization of Theorem B.

\begin{theorem}\label{T:sufficient_sum_closed}
Let $A_{p,i}:\mathcal{K}_{p,i}\to\mathcal{H}$, $i=1,\ldots,N_p$, $p=1,\ldots,m$ be bounded linear operators
(here $\mathcal{K}_{p,i}$ are Hilbert spaces, $m,N_1,\ldots,N_m$ are natural numbers)
such that $A_{p,i}^* A_{q,j}$ is compact for any $p\neq q$, $i=1,\ldots,N_p$, $j=1,\ldots,N_q$.
If $\mathcal{R}_p=\sum_{i=1}^{N_p}Ran(A_{p,i})$ is closed for any $p=1,\ldots,m$, then
$\sum_{p=1}^m \mathcal{R}_p$ is closed.
\end{theorem}

\begin{corollary}
If $Ran(A_i)$ is closed for $i=1,\ldots,N$, then $\sum_{i=1}^N Ran(A_i)$ is closed.
\end{corollary}

Note that our proof of Theorem~\ref{T:sufficient_sum_closed} is not a generalization of the proof of Theorem B
(and we don't use Theorem B in the proof of Theorem~\ref{T:sufficient_sum_closed}).

\subsection{Notation}
In this paper we consider only complex Hilbert spaces usually denoted by the letters~$\mathcal{H},\mathcal{K}$.
The scalar product in $\mathcal{H}$ is denoted by $\langle\cdot,\cdot\rangle$,
and $\|\cdot\|$ stands for the corresponding norm, $\|x\|^2=\langle x,x\rangle$.
The identity operator on~$\mathcal{H}$ is denoted by~$I_\mathcal{H}$ or simply $I$ if it is clear which Hilbert space is being considered.
For a bounded linear operator $X:\mathcal{H}\to\mathcal{H}$, $\sigma(X)$ denotes the spectrum of the operator $X$.

\section{Auxiliary results and notions}

\subsection{The essential spectrum of a self-adjoint operator}

Following~\cite[Chapter 9]{Birman}, we recall the definition and some properties of the essential spectrum of a self-adjoint operator.

Let $A$ be a bounded self-adjoint operator in a complex Hilbert space $\mathcal{H}$.
The essential spectrum of $A$, $\sigma_e(A)$, is the set of all $\lambda\in\sigma(A)$ such that
either $\lambda$ is a limit point of $\sigma(A)$ or $Ker(A-\lambda I)$ is infinite-dimensional.
The set $\sigma_d(A)=\sigma(A)\setminus\sigma_e(A)$ is called the discrete spectrum of $A$.
Clearly, $\lambda\in\sigma_d(A)$ if and only if $\lambda$ is an isolated point of $\sigma(A)$ and $Ker(A-\lambda I)$ is finite-dimensional.

It terms of the spectral measure $E_A(\cdot)$ the essential spectrum of $A$ can be characterized as follows:
$\lambda\in\sigma_e(A)$ if and only if $E_A((\lambda-\varepsilon,\lambda+\varepsilon))\mathcal{H}$ is infinite-dimensional for any $\varepsilon>0$.

More convenient for us is a description of $\sigma_e(A)$ in terms of singular sequences.
Recall that a sequence $\{x_k\mid k\geqslant 1\}$, $x_k\in\mathcal{H}$, $k\geqslant 1$,
is called a singular sequence for $A$ at the point $\lambda$ if
\begin{enumerate}
\item
$x_k\to 0$ weakly as $k\to\infty$;
\item
$x_k$ does not converge to $0$ as $k\to\infty$;
\item
$(A-\lambda I)x_k\to 0$ as $k\to\infty$.
\end{enumerate}
Let us denote by $sing(A,\lambda)$ the set of all sequences $\{x_k\mid k\geqslant 1\}$ which are singular for $A$ at $\lambda$.

\begin{proposition}
$\lambda\in\sigma_e(A)$ if and only if $sing(A,\lambda)\neq\varnothing$.
\end{proposition}

\subsection{On the sum of operator ranges}

\begin{proposition}\label{P:sum_operator_ranges}
Let $\mathcal{K}_1,\ldots,\mathcal{K}_n$ and $\mathcal{H}$ be Hilbert spaces,
$B_i:\mathcal{K}_i\to\mathcal{H}$ a bounded linear operators, $i=1,\ldots,n$.
If $\sum_{i=1}^n Ran(B_i)=\mathcal{H}$, then
$\sum_{i=1}^n B_i B_i^*\geqslant\varepsilon I$ for some $\varepsilon>0$.
\end{proposition}
\begin{proof}
Define an operator $B:\mathcal{K}_1\oplus\ldots\oplus\mathcal{K}_n\to\mathcal{H}$ by
\begin{equation*}
B(x_1,\ldots,x_n)=\sum_{i=1}^n B_i x_i,\quad x_i\in\mathcal{K}_i,\quad i=1,\ldots,n.
\end{equation*}
Then $B^*:\mathcal{H}\to\mathcal{K}_1\oplus\ldots\oplus\mathcal{K}_n$ and
\begin{equation*}
B^*x=(B_1^*x,\ldots,B_n^*x),\quad x\in\mathcal{H}.
\end{equation*}
Since $Ran(B)=\sum_{i=1}^n Ran(B_i)=\mathcal{H}$, we conclude that $B^*$ is an isomorphic embedding, that is,
there exists $c>0$ such that $\|B^*x\|\geqslant c\|x\|$, $x\in\mathcal{H}$.
We have
\begin{equation*}
\|B^*x\|^2=\sum_{i=1}^n\|B_i^*x\|^2=\sum_{i=1}^n\langle B_i B_i^*x,x\rangle=\left\langle\left(\sum_{i=1}^n B_i B_i^*\right)x,x\right\rangle.
\end{equation*}
Hence, $\sum_{i=1}^n B_i B_i^*\geqslant c^2 I$.
\end{proof}

We will need a simple corollary of Proposition~\ref{P:sum_operator_ranges}.

\begin{corollary}\label{C:sum operator ranges}
Let $\mathcal{K}_1,\ldots,\mathcal{K}_n$ and $\mathcal{H}$ be Hilbert spaces,
$B_i:\mathcal{K}_i\to\mathcal{H}$ a bounded linear operators, $i=1,\ldots,n$.
If $\sum_{i=1}^n Ran(B_i)$ is closed, then
\begin{equation*}
\sum_{i=1}^n Ran(B_i)=Ran\left(\sum_{i=1}^n B_i B_i^*\right).
\end{equation*}
\end{corollary}

The following proposition is well-known.

\begin{proposition}\label{P:operator range}
Let $B$ be a bounded self-adjoint operator in $\mathcal{H}$.
$Ran(B)$ is closed if and only if $\sigma(B)\cap((-\varepsilon,0)\cup(0,\varepsilon))=\varnothing$ for some $\varepsilon>0$.
\end{proposition}

The proof is trivial and is omitted.

\section{Proof of Theorem A}\label{S:proof_theorem_A}

To prove Theorem A, we need the following simple lemma
which shows relation between singular sequences of two bounded self-adjoint operators whose product is compact.

\begin{lemma}\label{L:lemma_A}
Let $B_1,B_2$ be bounded self-adjoint operators in $\mathcal{H}$.
Suppose $B_1 B_2$ is compact.
If $\{x_k\mid k\geqslant 1\}\in sing(B_2,\lambda)$, where $\lambda\neq 0$, then $\{x_k\mid k\geqslant 1\}\in sing(B_1,0)$.
\end{lemma}
\begin{proof}
We have $(B_2-\lambda I)x_k\to 0$, $k\to\infty$.
It follows that $B_1(B_2-\lambda I)x_k\to 0$, $k\to\infty$, that is, $B_1 B_2x_k-\lambda B_1 x_k\to 0$, $k\to\infty$.
Since $B_1 B_2$ is compact and $x_k\to 0$ weakly, we conclude that $B_1 B_2 x_k\to 0$.
Hence, $\lambda B_1 x_k\to 0$, $k\to\infty$, whence, $B_1 x_k\to 0$, $k\to\infty$.
Therefore, $\{x_k\mid k\geqslant 1\}\in sing(B_1,0)$.
\end{proof}

\begin{proof}[Proof of Theorem A]
\textbf{1.}
Suppose that $\lambda\in\sigma_e(A_i)$ for some $i$, and $\lambda\neq 0$.
Let us show that $\lambda\in\sigma_e(A)$.
There exists a sequence $\{x_k\mid k\geqslant 1\}\in sing(A_i,\lambda)$.
By Lemma~\ref{L:lemma_A}, $\{x_k\mid k\geqslant 1\}\in sing(A_j,0)$ for $j\neq i$.
We have
\begin{equation*}
(A-\lambda I)x_k=(A_i-\lambda I)x_k+\sum_{j\neq i}A_j x_k\to 0,\quad k\to\infty.
\end{equation*}
Hence, $\{x_k\mid k\geqslant 1\}\in sing(A,\lambda)$, whence $\lambda\in\sigma_e(A)$.

\textbf{2.}
Suppose that $\lambda\in\sigma_e(A)$, $\lambda\neq 0$.
Let us prove that $\lambda\in\sigma_e(A_i)$ for some $i$.
There exists $\{x_k\mid k\geqslant 1\}\in sing(A,\lambda)$.
We have
\begin{equation}\label{E:convergence_A}
(A-\lambda I)x_k=\left(\sum_{j=1}^N A_j-\lambda I\right)x_k\to 0,\quad k\to\infty.
\end{equation}
Let $i\in\{1,\ldots,n\}$.
Then $A_i(\sum_{j=1}^N A_j-\lambda I)x_k\to 0$, $k\to\infty$, that is,
$\sum_{j=1}^N A_i A_j x_k-\lambda A_i x_k\to 0$, $k\to\infty$.
Since $A_i A_j$ is compact for $j\neq i$ and $x_k\to 0$ weakly, we conclude that $A_i A_j x_k\to 0$, $k\to\infty$ for $j\neq i$.
Hence, $A_i^2 x_k-\lambda A_i x_k\to 0$, $k\to\infty$, i.e., $(A_i-\lambda I)A_i x_k\to 0$, $k\to\infty$.

Suppose that there exists $i$ such that the sequence $A_i x_k$ does not converge to $0$ as $k\to\infty$.
Then $\{A_i x_k\mid k\geqslant 1\}\in sing(A_i,\lambda)$, whence $\lambda\in\sigma_e(A_i)$.

Now assume that $A_i x_k\to 0$, $k\to\infty$ for any $i=1,\ldots,N$.
From~\eqref{E:convergence_A} it follows that $\lambda x_k\to 0$, $k\to\infty$.
Hence, $x_k\to 0$, $k\to\infty$, a contradiction.

The proof is complete.
\end{proof}

\section{Proof of Main Theorem}

To prove Main Theorem, we need the following two lemmas.

\begin{lemma}\label{L:compact_product}
Let $B,C$ be bounded self-adjoint operators in $\mathcal{H}$.
If $BC$ is compact, then
\begin{equation*}
E_B(\mathbb{R}\setminus[-\varepsilon,\varepsilon])E_C(\mathbb{R}\setminus[-\delta,\delta])
\end{equation*}
is compact for any $\varepsilon, \delta>0$.
\end{lemma}
\begin{proof}
Set $P=E_B(\mathbb{R}\setminus[-\varepsilon,\varepsilon])$, $Q=E_C(\mathbb{R}\setminus[-\delta,\delta])$.
Then $B^2\geqslant\varepsilon^2 P$, $C^2\geqslant\delta^2 Q$.

Since $BC$ is compact, we see that $BC^2B$ is compact.
Clearly, $BC^2B\geqslant\delta^2 BQB$.
Hence, $BQB$ is compact.
But $BQB=(BQ)(BQ)^*$.
Consequently, $BQ$ is compact.
Then $QB^2Q$ is compact.
Clearly, $QB^2Q\geqslant\varepsilon^2 QPQ$.
Hence, $QPQ$ is compact.
But $QPQ=(QP)(QP)^*$.
Consequently, $QP$ is compact.
It follows that $PQ=(QP)^*$ is compact.
\end{proof}

\begin{lemma}\label{L:gap_in_spectrum}
Let $P_1,\ldots,P_n$ be orthogonal projections in $\mathcal{H}$.
Suppose that $P_i P_j$ is compact for any $i\neq j$.
Then there exists $\varepsilon>0$ such that
\begin{equation*}
\sigma(\sum_{i=1}^n P_i)\cap(0,\varepsilon)=\varnothing.
\end{equation*}
\end{lemma}
\begin{proof}
Set $\mathcal{K}_i=Ran(P_i)$, $i=1,\ldots,n$.
Define an operator $\Gamma:\mathcal{H}\to\mathcal{K}_1\oplus\ldots\oplus\mathcal{K}_n$ by
\begin{equation*}
\Gamma x=(P_1 x,\ldots,P_n x),\quad x\in\mathcal{H}.
\end{equation*}
Then $\Gamma^*:\mathcal{K}_1\oplus\ldots\oplus\mathcal{K}_n\to\mathcal{H}$ and
\begin{equation*}
\Gamma^*(x_1,\ldots,x_n)=x_1+\ldots+x_n,\quad x_i\in\mathcal{K}_i,\quad i=1,\ldots,n.
\end{equation*}
Hence, $\Gamma^*\Gamma=\sum_{i=1}^n P_i$.
Clearly, the operator $\Gamma\Gamma^*:\mathcal{K}_1\oplus\ldots\oplus\mathcal{K}_n\to\mathcal{K}_1\oplus\ldots\oplus\mathcal{K}_n$ and
its block decomposition is equal to
\begin{equation*}
\Gamma\Gamma^*=(P_i\upharpoonright_{\mathcal{K}_j}:\mathcal{K}_j\to\mathcal{K}_i\mid 1\leqslant i,j\leqslant n).
\end{equation*}
Since $P_i P_j$ is compact for any $i\neq j$, we conclude that
\begin{equation*}
(\Gamma\Gamma^*)_{i,j}=P_i\upharpoonright_{\mathcal{K}_j}=P_i P_j\upharpoonright_{\mathcal{K}_j}
\end{equation*}
is compact for $i\neq j$. Consequently, $\Gamma\Gamma^*-I$ is compact.
By the Weyl theorem, $\sigma_e(\Gamma\Gamma^*)=\sigma_e(I)\subset\{1\}$.
Hence, $0\notin\sigma_e(\Gamma\Gamma^*)$, whence $\sigma(\Gamma\Gamma^*)\cap(0,\varepsilon)=\varnothing$ for some $\varepsilon>0$.
Since $\sigma(\Gamma^*\Gamma)\setminus\{0\}=\sigma(\Gamma\Gamma^*)\setminus\{0\}$, we conclude that
$\sigma(\sum_{i=1}^n P_i)\cap(0,\varepsilon)=\varnothing$.
\end{proof}

\begin{proof}[Proof of Main Theorem]
\textbf{1.}
Suppose the subspace $\mathcal{H}_\varepsilon$ is infinite dimensional for any $\varepsilon>0$.
We claim that $0\in\sigma_e(A)$.
To prove this, we construct $\{x_k\mid k\geqslant 1\}\in sing(A,0)$ as follows.
Take
\begin{equation*}
x_1\in\mathcal{H}_1,\quad \|x_1\|=1.
\end{equation*}
Suppose $x_1,\ldots,x_k$ have already been defined.
Clearly, we can choose
\begin{equation*}
x_{k+1}\in\mathcal{H}_{1/(k+1)},\quad \|x_{k+1}\|=1
\end{equation*}
such that $x_{k+1}$ is orthogonal to $x_i$, $i=1,\ldots,k$.
Hence, we obtain the sequence $\{x_k\mid k\geqslant 1\}$.
By the construction, $\{x_k\mid k\geqslant 1\}$ is orthonormal.
Moreover, $\|A_i x_k\|\leqslant 1/k$, $i=1,\ldots,N$.
It follows that $\|Ax_k\|\leqslant N/k$, $k\geqslant 1$.
Consequently, $Ax_k\to 0$ as $k\to\infty$.
Hence, $\{x_k\mid k\geqslant 1\}\in sing(A,0)$, whence $0\in\sigma_e(A)$.

\textbf{2.}
Suppose $\mathcal{H}_\varepsilon$ is finite dimensional for some $\varepsilon>0$.
Let us show that $0\notin\sigma_e(A)$.
Define $P_i=E_{A_i}(\mathbb{R}\setminus[-\varepsilon,\varepsilon])$, $i=1,\ldots,N$.
By Lemma~\ref{L:compact_product}, $P_i P_j$ is compact for any $i\neq j$.
By Lemma~\ref{L:gap_in_spectrum}, there exists $\delta>0$ such that $\sigma(\sum_{i=1}^N P_i)\cap(0,\delta)=\varnothing$.
It follows that
\begin{equation*}
\sum_{i=1}^N P_i+\delta Q\geqslant\delta I,
\end{equation*}
where $Q$ is the orthogonal projection onto $Ker(\sum_{i=1}^N P_i)$.
Clearly,
\begin{equation*}
Ker\left(\sum_{i=1}^N P_i\right)=\bigcap_{i=1}^N Ker(P_i)=\bigcap_{i=1}^N E_{A_i}([-\varepsilon,\varepsilon])\mathcal{H}=\mathcal{H}_\varepsilon.
\end{equation*}
Hence, $Q$ is a finite dimensional orthogonal projection.
Since $A_i^2\geqslant\varepsilon^2 P_i$, $i=1,\ldots,N$, we have
\begin{equation*}
\sum_{i=1}^N A_i^2+\varepsilon^2\delta Q\geqslant\sum_{i=1}^N \varepsilon^2 P_i+\varepsilon^2\delta Q\geqslant\varepsilon^2\delta I.
\end{equation*}
Hence,
\begin{equation}\label{E:ineq_1}
\sum_{i=1}^N A_i^2+\mu Q\geqslant\mu I,
\end{equation}
where $\mu=\varepsilon^2\delta$.

Now we are in a position to prove that $0\notin\sigma_e(A)$.
Suppose that $0\in\sigma_e(A)$.
There exists a sequence $\{x_k\mid k\geqslant 1\}\in sing(A,0)$.
We have $(\sum_{j=1}^N A_j)x_k\to 0$ as $k\to\infty$.
Let $i\in\{1,\ldots,N\}$.
Then $A_i(\sum_{j=1}^N A_j)x_k\to 0$ as $k\to\infty$, that is, $\sum_{j=1}^N A_i A_j x_k\to 0$ as $k\to\infty$.
Since $A_i A_j$ is compact for $j\neq i$ and $x_k\to 0$ weakly, we conclude that $A_i A_j x_k\to 0$ as $k\to\infty$ for $j\neq i$.
Consequently, $A_i^2 x_k\to 0$, $k\to\infty$.
Since $\{x_k\mid k\geqslant 1\}$ is bounded, we conclude that $\langle A_i^2 x_k,x_k\rangle\to 0$, $k\to\infty$.
Hence, $\|A_i x_k\|^2\to 0$, $k\to\infty$, whence $A_i x_k\to 0$, $k\to\infty$.
Since $Q$ is a finite dimensional operator and $x_k\to 0$ weakly, we conclude that $Qx_k\to 0$, $k\to\infty$.
From~\eqref{E:ineq_1} it follows that
\begin{equation*}
\sum_{i=1}^N \langle A_i^2 x_k,x_k\rangle+\mu\langle Qx_k,x_k\rangle\geqslant\mu\|x_k\|^2,
\end{equation*}
that is,
\begin{equation*}
\sum_{i=1}^N\|A_i x_k\|^2+\mu\|Qx_k\|^2\geqslant\mu\|x_k\|^2.
\end{equation*}
It follows that $x_k\to 0$ as $k\to\infty$, a contradiction.

Hence, $0\notin\sigma_{e}(A)$.
\end{proof}

\section{Proof of Theorem~\ref{T:criterion_sum_closed}}

\begin{proof}[Proof of Theorem~\ref{T:criterion_sum_closed}]
Let $\mathcal{K}=\mathcal{H}\ominus\mathcal{H}_0$, then $\mathcal{H}=\mathcal{H}_0\oplus\mathcal{K}$.
With respect to this orthogonal decomposition $A_i=0\oplus B_i$, $i=1,\ldots,N$, where $B_i$ is a bounded self-adjoint operator in $\mathcal{K}$.
Since $\mathcal{H}_0=\cap_{i=1}^N Ker(A_i)$, we see that $\cap_{i=1}^N Ker(B_i)=\{0\}$.
Since $A_i A_j$ is compact for $i\neq j$, we conclude that $B_i B_j$ is compact for $i\neq j$.

\textbf{1.}
Suppose $\sum_{i=1}^N Ran(A_i)$ is closed. Let us show that $\mathcal{H}_\delta=\mathcal{H}_0$ for some $\delta>0$.

$\sum_{i=1}^N Ran(B_i)=\sum_{i=1}^N Ran(A_i)$ is closed.
Since
\begin{equation*}
\mathcal{K}\ominus\left(\sum_{i=1}^N Ran(B_i)\right)=\bigcap_{i=1}^N (\mathcal{K}\ominus Ran(B_i))=\bigcap_{i=1}^N Ker(B_i)=\{0\},
\end{equation*}
we conclude that $\sum_{i=1}^N Ran(B_i)=\mathcal{K}$.
By Proposition~\ref{P:sum_operator_ranges}, $\sum_{i=1}^N B_i^2\geqslant\varepsilon I$ for some $\varepsilon>0$.
Set $\delta=\sqrt{\varepsilon/(2N)}$.
We claim that $\mathcal{H}_\delta=\mathcal{H}_0$.
Let us prove this.
We have $E_{A_i}(\cdot)=E_0(\cdot)\oplus E_{B_i}(\cdot)$.
Hence, $E_{A_i}([-\delta,\delta])=I\oplus E_{B_i}([-\delta,\delta])$, whence
$E_{A_i}([-\delta,\delta])\mathcal{H}=\mathcal{H}_0\oplus E_{B_i}([-\delta,\delta])\mathcal{K}$.
Thus
\begin{equation*}
\mathcal{H}_\delta=\mathcal{H}_0\oplus\cap_{i=1}^N E_{B_i}([-\delta,\delta])\mathcal{K}.
\end{equation*}
But $\cap_{i=1}^N E_{B_i}([-\delta,\delta])\mathcal{K}=\{0\}$.
Indeed, suppose $x\in\cap_{i=1}^N E_{B_i}([-\delta,\delta])\mathcal{K}$ and $x\neq 0$.
Then $\|B_i x\|\leqslant\delta\|x\|$, $i=1,\ldots,N$.
Hence,
\begin{equation*}
\left\langle\left(\sum_{i=1}^N B_i^2\right)x,x\right\rangle=\sum_{i=1}^N\|B_i x\|^2\leqslant N\delta^2\|x\|^2=\frac{\varepsilon}{2}\|x\|^2.
\end{equation*}
But $\langle(\sum_{i=1}^N B_i^2)x,x\rangle\geqslant\varepsilon\|x\|^2$.
We get a contradiction.
Hence, $\cap_{i=1}^N E_{B_i}([-\delta,\delta])\mathcal{K}=\{0\}$, and, consequently, $\mathcal{H}_\delta=\mathcal{H}_0$.

\textbf{2.}
Let us prove that if $\mathcal{H}_\varepsilon=\mathcal{H}_0$ for some $\varepsilon>0$, then $\sum_{i=1}^N Ran(A_i)$ is closed.
We will prove a more stronger fact: if $\mathcal{H}_\varepsilon\ominus\mathcal{H}_0$ is finite dimensional for some $\varepsilon>0$, then
$\sum_{i=1}^N Ran(A_i)$ is closed.

Since $A_i^2=0\oplus B_i^2$, we see that $E_{A_i^2}(\cdot)=E_0(\cdot)\oplus E_{B_i^2}(\cdot)$.
Hence, $E_{A_i^2}([-\varepsilon^2,\varepsilon^2])=I\oplus E_{B_i^2}([-\varepsilon^2,\varepsilon^2])$.
But $E_{A_i^2}([-\varepsilon^2,\varepsilon^2])=E_{A_i}([-\varepsilon,\varepsilon])$.
Hence, $E_{A_i}([-\varepsilon,\varepsilon])=I\oplus E_{B_i^2}([-\varepsilon^2,\varepsilon^2])$, whence
$E_{A_i}([-\varepsilon,\varepsilon])\mathcal{H}=\mathcal{H}_0\oplus E_{B_i^2}([-\varepsilon^2,\varepsilon^2])\mathcal{K}$.
Thus
\begin{equation*}
\mathcal{H}_\varepsilon=\mathcal{H}_0\oplus\cap_{i=1}^N E_{B_i^2}([-\varepsilon^2,\varepsilon^2])\mathcal{K}.
\end{equation*}
Thus, $\cap_{i=1}^N E_{B_i^2}([-\varepsilon^2,\varepsilon^2])\mathcal{K}$ is finite dimensional.
By the Main Theorem, $0\notin\sigma_{e}(\sum_{i=1}^N B_i^2)$.
Moreover, since
\begin{equation*}
Ker\left(\sum_{i=1}^N B_i^2\right)=\bigcap_{i=1}^N Ker(B_i)=\{0\},
\end{equation*}
we conclude that $0\notin\sigma_d(\sum_{i=1}^N B_i^2)$.
Hence, $0\notin\sigma(\sum_{i=1}^N B_i^2)$, that is, the operator $\sum_{i=1}^N B_i^2$ is invertible.
It follows that $\sum_{i=1}^N Ran(B_i)=\mathcal{K}$.
Consequently, $\sum_{i=1}^N Ran(A_i)=\mathcal{K}$ is closed.
\end{proof}

\section{Proof of Theorem~\ref{T:sufficient_sum_closed}}

\begin{proof}[Proof of Theorem~\ref{T:sufficient_sum_closed}]
Define
\begin{equation*}
B_p=\sum_{i=1}^{N_p}A_{p,i}A_{p,i}^*,\quad p=1,\ldots,m.
\end{equation*}
Clearly, $B_p$ is a bounded self-adjoint operator in $\mathcal{H}$.
Moreover, $B_p B_q$ is compact for $p\neq q$.

From Corollary~\ref{C:sum operator ranges} it follows that $Ran(B_p)=\mathcal{R}_p$, $p=1,\ldots,m$.
By Proposition~\ref{P:operator range}, there exists $\varepsilon>0$ such that
$\sigma(B_p)\cap((-\varepsilon,0)\cup(0,\varepsilon))=\varnothing$ for $p=1,\ldots,m$.
It follows that $E_{B_p}([-\varepsilon/2,\varepsilon/2])=E_{B_p}(\{0\})$, and, consequently,
$E_{B_p}([-\varepsilon/2,\varepsilon/2])\mathcal{H}=Ker(B_p)$, $p=1,\ldots,m$.
Hence,
\begin{equation*}
\mathcal{H}_{\varepsilon/2}(B_1,\ldots,B_m)=\bigcap_{p=1}^m E_{B_p}([-\varepsilon/2,\varepsilon/2])\mathcal{H}=
\bigcap_{p=1}^m Ker(B_p)=\mathcal{H}_0(B_1,\ldots,B_m).
\end{equation*}
By Theorem~\ref{T:criterion_sum_closed}, $\sum_{p=1}^m Ran(B_p)=\sum_{p=1}^m\mathcal{R}_p$ is closed.
\end{proof}

{\bf Acknowledgement.}
The author is grateful to Alexei Yu. Konstantinov for helpful suggestions.

\bibliographystyle{amsplain}

\begin{thebibliography}{99}

\bibitem{Badea}
C. Badea, S. Grivaux, V. Muller,
\textit{The rate of convergence in the method of alternating projections},
St. Petersburg Math. J.
\textbf{23} (2012), no. 3,
413--434.

\bibitem{Birman}
M. S. Birman, M. Z. Solomjak,
\textit{Spectral theory of self-adjoint operators in Hilbert space},
D. Reidel Publishing Co., Dordrecht, 1987.

\bibitem{Deutsch}
F. Deutsch,
\textit{Best approximation in inner product spaces},
CMS Books Math.,
Springer-Verlag,
New York,
2001.

\bibitem{Feshchenko}
I. S. Feshchenko,
\textit{On closeness of the sum of $n$ subspaces of a Hilbert space},
Ukrainian Math. J.
\textbf{63} (2012), no. 10,
1566--1622.

\bibitem{Lang}
H. Lang,
\textit{On sums of subspaces in topological vector spaces and an application in theoretical tomography},
Appl. Anal.
\textbf{18} (1984), no. 4,
257--265.

\bibitem{Peller}
V. V. Peller,
\textit{Hankel operators and their applications},
Springer Monogr. Math,
2002.

\bibitem{Pushnitski_Yafaev}
A. Pushnitski, D. Yafaev,
\textit{A multichannel scheme in smooth scattering theory},
J. Spectr. Theory,
arXiv:1209.3238v1 [math.SP] (to appear).

\bibitem{Svensson report}
L. Svensson,
\textit{On the closure of sums of closed subspaces},
Research report,
Royal Inst. of Technology (Stockholm),
1981.

\bibitem{Svensson article}
L. Svensson,
\textit{Sums of complemented subspaces in locally convex spaces},
Ark. Mat.
\textbf{25} (1987), no. 1,
147--153.

\end{thebibliography}

\end{document}